\documentclass[11pt]{article}
\usepackage{amsmath,amssymb,amsthm,amsfonts}
\usepackage[pdftex]{color,graphicx}
\usepackage[utf8]{inputenc} 
\usepackage{graphics}
\usepackage{enumerate}
\usepackage{array}
\usepackage[english]{babel}
\usepackage{verbatim}
\usepackage{hyperref}
\usepackage{caption}
\usepackage{tikz-cd}

\usepackage{pdflscape}
\usepackage{rotating}
\usepackage{graphicx}

\usepackage{geometry} % to change the page dimensions
\usepackage{amsmath}
\usepackage{amsfonts}
\usepackage{amssymb}
\usepackage{amsthm}
\usepackage{mathrsfs}
\usepackage{mathtext}
\usepackage{mathtools}
\usepackage[inline]{enumitem}
\usepackage[all,cmtip]{xy}
\usepackage{authblk}

\newcommand{\ot}{\otimes}

\DeclareMathOperator{\coh}{H}
\DeclareMathOperator{\Ext}{Ext}
\DeclareMathOperator{\HH}{HH}

\numberwithin{equation}{section}

\newtheorem{theo}[equation]{Theorem}

\newtheorem{lemma}[equation]{Lemma}

\title{Gerstenhaber bracket on Hopf algebra and Hochschild cohomologies}

\author{Tekin Karada\u{g}}

\newcommand{\Addresses}{{% additional braces for segregating \footnotesize
  \bigskip
  \footnotesize

  \textsc{Department of Mathematics, Texas A\&M University, 
College Station, Texas 77843, USA}\par\nopagebreak
  \textit{E-mail address}: \texttt{tekinkaradag@math.tamu.edu}
}}

\date{\today}

\begin{document}

\maketitle
\thispagestyle{empty}
%\clearpage

\begin{abstract}
We calculate the Gerstenhaber bracket on Hopf algebra and Hochschild cohomologies of the Taft algebra $T_p$ for any integer $p> 2$ which is a nonquasi-triangular Hopf algebra. We show that the bracket is indeed zero on Hopf algebra cohomology of $T_p$, as in all known quasi-triangular Hopf algebras. This example is the first known bracket computation for a nonquasi-triangular algebra. Also, we find a general formula for the bracket on Hopf algebra cohomology of any Hopf algebra with bijective antipode on the bar resolution that is reminiscent of Gerstenhaber's original formula for Hochschild cohomology.
\end{abstract}

\let\thefootnote\relax\footnote{{\em Key words and phrases: } Hochschild cohomology, Hopf algebra cohomology, Gerstenhaber bracket, Taft algebra}
\let\thefootnote\relax\footnote{Partially supported by NSF grant 1665286.}
\section{Introduction}\label{sec:introduction}

Gerstenhaber brackets were originally defined on Hochschild cohomology by M. Gerstenhaber himself \cite[Section 1.1]{G}. In 2002, A. Farinati and A. Solotar showed that for any Hopf algebra $A$, Hopf algebra cohomology H$^*(A):=\Ext^*_A(k,k)$ is a Gerstenhaber algebra \cite{FAR}. Hence, we can define a Gerstenhaber bracket on Hopf algebra cohomology. In the same year, R. Taillefer used a different approach and found a bracket on Hopf algebra cohomology \cite{Tail} which is equivalent to the bracket constructed by A. Farinati and A. Solotar. The category of $A$-modules and the category of $A^e$-modules are examples of strong exact monoidal categories. In 2016, Reiner Hermann \cite[Theorem 6.3.12, Corollary 6.3.15]{Hermann} proved that if the strong exact monoidal category is lax braided, then the bracket is constantly zero. Therefore, the Gerstenhaber bracket on the Hopf algebra cohomology of a quasi-triangular Hopf algebra is trivial. However, we do not know the bracket structure for a nonquasi-triangular Hopf algebra. Taft algebras are nice examples of nonquasi-triangular Hopf algebras. In this paper, we show that the Gerstenhaber bracket on the Hochschild cohomology of a Taft Algebra is nontrivial. However, the bracket structure on Hopf algebra cohomology of a Taft algebra is constantly zero. Also, we take the Gerstenhaber bracket formula on Hochschild comology and find a general formula for Gerstenhaber bracket on Hopf algebra cohomology.

We start by giving some basic definitions and some tools to calculate the bracket on Hochschild cohomology in Section 2. Then, we compute the Gerstenhaber bracket on the Hochschild cohomology of $A=k[x]/(x^p)$ where the field $k$ has characteristic 0 and the integer $p>2$ in Section 3. We use the technique introduced by C. Negron and S. Witherspoon \cite{NW} and note that they computed the bracket on Hochschild cohomology of $A$ for the case that $k$ has positive characteristic $p$ \cite[Section 5]{NW}. 

In Section 4, we compute the Gerstenhaber bracket for the Taft algebra $T_p$ which is a nonquasi-triangular Hopf algebra. We use a similar technique as in \cite{NW} to calculate the bracket on Hochschild cohomology of $T_p$. It is also known that the Hopf algebra cohomology of any Hopf algebra with a bijective antipode can be embedded in the Hochschild cohomology of the algebra \cite[Theorem 9.4.5 and Corollary 9.4.7]{HH}. Since all finite dimensional Hopf algebras (also most of known infinite dimensional Hopf algebras) have bijective antipode, we can embed the Hopf algebra cohomology of $T_p$ into the Hochschild cohomology of $T_p$. Then, we use this explicit embedding and find the bracket on the Hopf algebra cohomology of $T_p$. As a result of our calculation, we obtain that the bracket on Hopf algebra cohomology of $T_p$ is also trivial. 

In the last section, we derive a general expression for the bracket on Hopf algebra cohomology of any Hopf algebra $A$ with bijective antipode. We first consider a specific resolution that agrees with the bar resolution of $A$ and find a bracket formula for it. Then, we use the composition of various isomorphisms and an embedding from Hopf algebra cohomology into Hochschild cohomology in order to discover the bracket formula on Hopf algebra cohomology.

\section{Gerstenhaber Bracket on Hochschild Cohomology}

Let $k$ be a field, $A$ be a $k$-algebra, and $A^e=A\ot_k A^{op}$ where $A^{op}$ is the opposite algebra with reverse multiplication. For simplicity, we write $\ot$ instead of $\ot_k$. The following resolution $B(A)$ is a free resolution of the $A^e$-module $A$, called the \textit{bar resolution},
\begin{equation}\label{bar complex}
B(A):\cdots \stackrel{d_3}{\longrightarrow} A^{\otimes 4} \stackrel{d_2}{\longrightarrow} A^{\otimes 3} \stackrel{d_1}{\longrightarrow} A^{\otimes 2} \stackrel{\pi}{\longrightarrow} A \longrightarrow 0,
\end{equation}
where 
$$d_n(a_0\ot a_1\ot \cdots \ot a_{n+1})=\sum_{i=0}^{n}(-1)^ia_0\ot a_1\ot \cdots\ot a_ia_{i+1}\ot \cdots \ot a_{n+1}$$ 
and $\pi$ is multiplication.

Consider the following complex that is derived by applying Hom$_{A^e}(-,A)$ to the bar resolution $B(A)$
\begin{equation}\label{hom bar}
0 {\longrightarrow} \text{Hom}_{A^e}(A^{\otimes 2},A) \stackrel{d^{*}_1}{\longrightarrow} \text{Hom}_{A^e}(A^{\otimes 3},A) \stackrel{d^{*}_2}{\longrightarrow} \text{Hom}_{A^e}(A^{\otimes 4},A) \stackrel{d^{*}_3}{\longrightarrow} \cdots
\end{equation}
where $d_n^{*}(f)=fd_n$. The \textit{Hochshild cohomology} of the algebra $A$ is the cohomology of the cochain complex (\ref{bar complex}), i.e. $$\text{HH}^{*}(A,A)=\bigoplus_{n\geq 0}\Ext_{A^e}^n(A,A).$$

We also define the \textit{Hopf algebra cohomology} of the Hopf algebra $A$ over the field $k$ as $$\textup{H}^{*}(A,k)=\bigoplus_{n\geq 0}\textup{Ext}^n_A(k,k)$$ under the cup product.

Let $f\in \text{Hom}_k(A^{\otimes m},A)$ and $g\in \text{Hom}_k(A^{\otimes n},A)$. The Hochschild cohomology of $A$ is an algebra with the following cup product and the Gerstenhaber bracket structures. The \textit{cup product} $f\smile g\in$Hom$_k(A^{\otimes(m+n)},A)$ is defined by
$$(f\smile g)(a_1\otimes \cdots \otimes a_{m+n}):=(-1)^{mn}f(a_1\otimes \cdots \otimes a_m)g(a_{m+1}\otimes\cdots a_{m+n})$$
for all $a_1,\cdots, a_{m+n}\in A$, and the \textit{Gerstenhaber bracket} $[f,g]$ is an element of\\
Hom$_k(A^{\otimes(m+n-1)},A)$ given by
$$[f, g] := f\circ g-(-1)^{(m-1)(n-1)}g\circ f$$
where the circle product $f\circ g$ is
\begin{align*}
&(f\circ g)(a_1\otimes \cdots \otimes a_{m+n-1}):=\\
&\sum_{i=1}^{m}(-1)^{(n-1)(i-1)}f(a_1\otimes\cdots a_{i-1}\otimes g(a_i\otimes \cdots a_{i+n-1})\otimes a_{i+n}\otimes \cdots \otimes a_{m+n-1})
\end{align*}
for all $a_1,\cdots, a_{m+n-1}\in A$. We note that these definitions directly come from the bar resolution.

There is an identity between cup product and bracket \cite[Section 1]{G}:
\begin{equation}\label{cupbrac}
[f^*\smile g^*,h^*]=[f^*,h^*]\smile g^*+(-1)^{|f^*|(|h^*|-1)}f^*\smile [g^*,h^*],
\end{equation}
where $f^*, g^*,$ and $h^*$ are the images (in Hochschild cohomology) of the cocyles $f,g$, and $h$, respectively.

Computing the bracket on the bar resolution is not an ideal method. Instead, we can use another resolution, $\mathbb{A}\stackrel{\mu}\rightarrow A$,  satisfying the following hypotheses \cite[(3.1) and Lemma 3.4.1]{NW}:

(a) $\mathbb{A}$ admits an embedding $\iota:\mathbb{A}\to B(A)$  of complexes of $A$-bimodules for which
the following diagram commutes
\begin{center}
	\begin{tikzcd}[column sep=small]
		\mathbb{A} \arrow{r}{\iota}  \arrow{rd} 
		& B(A) \arrow{d} \\
		& A
	\end{tikzcd}
	
\end{center}

(b) The embedding $\iota$ admits a section $\pi:B\to \mathbb{A}$, i.e.\ an $A^e$-chain map $\pi$ with
$\pi\iota= id_{\mathbb{A}}$.

(c) There is a diagonal map that satisfies $\Delta_{\mathbb{A}}^{(2)}=(\pi\ot_A\pi\ot_A\pi)\Delta^{(2)}_{B(A)}\iota$ where $\Delta^{(2)}=(id\ot \Delta)\Delta$.

 We give the following theorem which is the combination of \cite[Theorem 3.2.5]{NW} and \cite[Lemma 3.4.1]{NW} that allows us to use a different resolution for the bracket calculation.
\begin{theo}\label{theo brac}
	Suppose $\mathbb{A}\stackrel{\mu}\rightarrow A$ is a projective $A$-bimodule resolution of $A$ that satisfies the hypotheses (a)-(c). Let $\phi:\mathbb{A}\otimes_A \mathbb{A}\to \mathbb{A}$ be any contracting homotopy for the chain map $F_\mathbb{A}:\mathbb{A}\ot_A \mathbb{A}\to \mathbb{A}$ defined by $F_\mathbb{A}:=(\mu\otimes_A id_\mathbb{A}- id_\mathbb{A}\otimes_A \mu)$, i.e.
	\begin{equation}\label{org phi}
	   d(\phi):=d_\mathbb{A}\phi+\phi d_{ \mathbb{A}\otimes_A \mathbb{A}}=F_\mathbb{A}.
	\end{equation} 
	Then for cocycles $f$ and $g$ in
	Hom$_{A^e}(\mathbb{A},A)$, the bracket given by 
	\begin{equation}\label{brac}
	[f, g]_\phi = f\circ_\phi g-(-1)^{(|f|-1)(|g|-1)}g\circ_\phi f
	\end{equation}
	where the circle product is 
	\begin{equation}  \label{circ}
	f\circ_\phi g=f\phi(id_\mathbb{A}\otimes_A g \otimes_A id_\mathbb{A})\Delta^{(2)}
	\end{equation}
	agrees with the Gerstenhaber bracket on cohomology.
\end{theo}

In general, it is not easy to calculate the map $\phi$ by the formula \eqref{org phi}. We use alternative way to find $\phi$.

Let $h$ be any $k$-linear contracting homotopy for the identity map on the extended complex $\mathbb{A}\to A\to 0$ where $\mathbb{A}$ is free. A contracting homotopy $\phi_i:(\mathbb{A}\otimes_A\mathbb{A})_i\longrightarrow\mathbb{A}_{i+1}$ in Theorem \ref{theo brac} is constructed by the following formula \cite[Lemma 3.3.1]{NW}: 
\begin{equation}\label{mapphi}
\phi_i=h_i((F_\mathbb{A})_i-\phi_{i-1}d_{(\mathbb{A}\otimes_A\mathbb{A})_i}).
\end{equation}

\section{Bracket on Hochschild cohomology of $A=k[x]/(x^p)$}
Let $A=k[x]/(x^p)$ where $k$ is a field of characteristic 0 and $p>2$ is an integer. We compute the Lie bracket on Hochschild cohomology of $A$ by Theorem \ref{theo brac}. We work on a smaller resolution of $A$ than the bar resolution of $A$. Consider the following $A^e$-module resolution of $A$:
\begin{equation}\label{res A}
\mathbb{A}:\cdots \stackrel{v.}{\longrightarrow} A^{e}\stackrel{u.}{\longrightarrow} A^{e} \stackrel{v.}{\longrightarrow} A^{e} \stackrel{u.}{\longrightarrow} A^{e} \stackrel{\pi}{\longrightarrow} A \longrightarrow 0,
\end{equation}
where $u=x\otimes 1-1\otimes x$, $v=x^{p-1}\otimes 1+x^{p-2}\ot x+\cdots+x\otimes x^{p-2}+1\otimes x^{p-1}$, and $\pi$ is the multiplication.

 The bracket on $A$ where $k$ is a field with positive characteristic, is calculated by C. Negron and S. Witherspoon \cite[Section 5]{NW}. We adopt the contracting homotopy $h$ for the identity map from that calculation and obtain a new map $h$ for our setup. Let $\xi_i$ be the element $1\otimes1$ of $\mathbb{A}_i$. The following maps $h_n:\mathbb{A}_n\longrightarrow \mathbb{A}_{n+1}$ form a contracting homotopy for identity map, as we can see by direct calculation:

\begin{align}\label{hn}
\begin{split}
    h_{-1}(x^i)&=\xi_0x^i,\\
h_{0}(x^i\xi_0x^j)&=\sum_{l=0}^{i-1}x^l\xi_1x^{i+j-1-l},\\
h_{1}(x^i\xi_1x^j)&=\delta_{i,p-1}x^j\xi_2,\\
h_{2n}(x^i\xi_{2n}x^j)&=-\sum_{l=0}^{j-1}x^{i+j-1-l}\xi_{2n+1}x^l \text{    $(n\geq 2)$  ,}\\
h_{2n+1}(x^i\xi_{2n+1}x^j)&=\delta_{j,p-1}x^i\xi_{2n+2} \text{    $(n\geq 2)$  .}
\end{split}
\end{align}

Then, we take $\phi_{-1}=0$ and construct the following $A^e$-linear maps $\phi_i:(\mathbb{A}\otimes_A\mathbb{A})_i\longrightarrow\mathbb{A}_{i+1}$ for degree 1 and 2 by (\ref{mapphi}):

\begin{align}\label{phin}
\begin{split}
    &\phi_0(\xi_0\otimes_A x^i\xi_0)=\sum_{l=0}^{i-1}x^l\xi_1x^{i-1-l},\\
&\phi_1(\xi_1\otimes_A x^i\xi_0)=-\delta_{i,p-1}\xi_2,\\
&\phi_1(\xi_0\otimes_A x^i\xi_1)=\delta_{i,p-1}\xi_2.
\end{split}
\end{align}

Lastly, we form the following diagonal map $\Delta:\mathbb{A}\longrightarrow\mathbb{A}\otimes_A\mathbb{A}$:
\begin{align}\label{Deltan}
\begin{split}
    &\Delta_0(\xi_0) = \xi_0\otimes_A\xi_0,\\
&\Delta_1(\xi_1) =\xi_1\otimes_A \xi_0 + \xi_0 \otimes_A \xi_1,\\
&\Delta_{2n}(\xi_{2n}) = \sum_{i=0}^{n}\xi_{2i} \otimes_A\xi_{2n-2i} +\sum_{i=0}^{n-1}\sum_{a+b+c=p-2}^{}x^a\xi_{2i+1} \otimes_A x^b\xi_{2n-2i-1}x^c, \text{ for }n\geq 1\\
&\Delta_{2n+1}(\xi_{2n+1}) = \sum_{i=0}^{2n+1}\xi_i \otimes_A \xi_{2n+1-i}, \text{ for }n\geq 1.
\end{split}
\end{align}
It can be seen that the map $\Delta$ is a chain map lifting the canonical isomorphism $A\stackrel{\sim}\rightarrow  A\ot_A A$ by direct calculation.

Now, we are ready to calculate the brackets on cohomology in low degrees. By applying Hom$_{A^e}(-,A)$
to $\mathbb{A}$, we see that the differentials are all 0 in odd degrees and $(px^{p-1})\cdot$ in even degrees. In each degree, the term in the Hom complex is the free $A$-module Hom$_{A^e}(A^e,A)\cong A$. Moreover, since $p$ is not divisible by the characteristic of $k$, we deduce HH$^0(A)\cong A, \text{HH}^{2i+1}(A)\cong (x),\text{ and }\text{HH}^{2i}(A)\cong A/(x^{p-1})$ \cite[Section 1.1]{HH}. 

Let $x^j\xi^*_i \in$ Hom$_{A^e}(A^e,A)$ denote the function that takes $\xi_i$ to $x^j$. Since the characteristic of $k$ does not divide $p$, the Hochschild cohomology as an $A$-algebra is generated by $\xi_1^{*}$ and $\xi_2^{*}$ \cite[Example 2.2.2]{HH}. We only calculate the brackets of the elements of degrees 1 and 2 which can be extended to higher degrees by the formula (\ref{cupbrac}). Hence, we have the following calculations:

The bracket of the elements of degrees 1 and 1:\\
\begin{align*}
&(x^i\xi_1^{*}\circ_{\phi}x^j\xi_1^{*})(\xi_1)\\
&=x^i\xi_1^{*}\phi(1\otimes_A x^j\xi_1^{*}\otimes_A1)\Delta^{(2)}(\xi_1)\\
&=x^i\xi_1^{*}\phi(1\otimes_Ax^j\xi_1^{*}\otimes_A1)(\xi_1\otimes_A\xi_0\otimes_A\xi_0+\xi_0\otimes_A\xi_1\otimes_A\xi_0+\xi_0\otimes_A\xi_0\otimes\xi_1)\\
&=x^i\xi_1^{*}\phi(\xi_0\otimes_A x^j\xi_0)\\
&=x^i\xi_1^{*}(\xi_1x^{j-1}+x\xi_1x^{j-2}+\cdots+x^{j-1}\xi_1)\\
&=jx^{i+j-1}
\end{align*}
and by symmetry
$(x^j\xi_1^{*}\circ_{\phi}x^i\xi_1^{*})(\xi_1)=ix^{i+j-1}$. Therefore, we have 
$$[x^i\xi_1^{*},x^j\xi_1^{*}]=(j-i)x^{i+j-1}\xi_1^{*}.$$
The bracket of the elements of degrees 1 and 2:\\
\begin{align*}
&(x^i\xi_1^{*}\circ_{\phi}x^j\xi_2^{*})(\xi_2)\\
&=x^i\xi_1^{*}\phi(1\otimes_A x^j\xi_2^{*}\otimes_A1)\Delta^{(2)}(\xi_2)\\
&=x^i\xi_1^{*}\phi(1\otimes_A x^j\xi_2^{*}\otimes_A1)(\xi_0\otimes_A\xi_0\otimes_A\xi_2+\xi_0\otimes_A\xi_2\otimes_A\xi_0+\xi_2\otimes_A \xi_0\otimes_A\xi_0\\
&+\xi_0\otimes_A\sum_{\substack{a+b+c \\ =p-2}}(x^a\xi_1\ot_A x^b\xi_1x^c)+\sum_{\substack{a+b+c \\ =p-2}}x^a\xi_1\ot_A x^b(\xi_0\ot_A \xi_1+\xi_1\ot_A x_0)x^c)\\
&=x^i\xi_1^{*}\phi(\xi_0\otimes_A x^j\xi_0)=x^i\xi_1^{*}(\xi_1x^{j-1}+x\xi_1x^{j-2}+\cdots+x^{j-1}\xi_1)=jx^{i+j-1}.
\end{align*}
The circle product in the reverse order is
\begin{align*}
&(x^j\xi_2^{*}\circ_{\phi}x^{p-1}\xi_1^{*})(\xi_2)\\
&=x^j\xi_2^{*}\phi(1\otimes_A x^{p-1}\xi_1^{*}\otimes_A1)\Delta^{(2)}(\xi_2)\\
&=x^j\xi_2^{*}\phi(1\otimes_A x^{p-1}\xi_1^{*}\otimes_A1)(\xi_0\otimes_A\xi_0\otimes_A\xi_2+\xi_0\otimes_A\xi_2\otimes_A\xi_0+\xi_2\otimes_A \xi_0\otimes_A\xi_0\\
&+\xi_0\otimes_A\sum_{\substack{a+b+c \\ =p-2}}(x^a\xi_1\ot_A x^b\xi_1x^c)+\sum_{\substack{a+b+c \\ =p-2}}x^a\xi_1\ot_A x^b(\xi_0\ot_A \xi_1+\xi_1\ot_A x_0)x^c)\\
&=x^j\xi_2^{*}\phi(\sum_{\substack{a+b+c \\ =p-2}}(\xi_0\otimes_A x^{a+b+i}\xi_1x^c-x^a\xi_1\ot_A x^{b+i}\xi_0x^c))\\
&=x^j\xi_2^{*}(\sum_{\substack{a+b+c \\ =p-2}}(\delta_{a+b+i,p-1}\xi_2x^c+x^a\delta_{b+i,p-1}\xi_2x^c))\\
&=x^j\xi_2^{*}((p-i)\xi_2x^{i-1}+\sum_{\substack{a+c \\ =i-1}}x^a\xi_2x^c)\\
&=(p-i)x^{i+j-1}+\sum_{\substack{a+c \\ =i-1}}x^{a+c+j}=(p-i)x^{i+j-1}+ix^{i+j-1}=px^{i+j-1}.
\end{align*}
Therefore, we obtain

$$[x^i\xi_1^{*},x^j\xi_2^{*}]=(j-p)x^{i+j-1}\xi^{*}_2.$$ 

Lastly, the bracket of the elements of degrees 2 and 2:\\
\begin{align*}
(x^i\xi_2^{*}\circ_{\phi}x^j\xi_2^{*})(\xi_3)&=x^i\xi_2^{*}\phi(1\otimes_A x^j\xi_2^{*}\otimes_A1)\Delta^{(2)}(\xi_3)\\
&=x^i\xi_2^{*}\phi(\xi_1\otimes_A x^j\xi_0+\xi_0\otimes_A x^j\xi_1)=x^i\xi_2^{*}(0)=0
\end{align*}
and by symmetry
$(x^j\xi_2^{*}\circ_{\phi}x^i\xi_2^{*})(\xi_3)=0$. Therefore, we have 
$$[(x^i\xi_2^{*},x^j\xi_2^{*})]=0.$$

As a consequence, the brackets for the elements of degrees 1 and 2 are 
\begin{align*}
[(x^i\xi_1^{*},x^j\xi_1^{*})]&=(j-i)x^{i+j-1}\xi_1^{*},\\
[(x^i\xi_1^{*},x^j\xi_2^{*})]&=(j-p)x^{i+j-1}\xi^{*}_2,\\
[(x^i\xi_2^{*},x^j\xi_2^{*})]&=0.
\end{align*}

Brackets in higher degrees can be determined from these and the identity (\ref{cupbrac})
since the Hochschild cohomology is generated as an $A$-algebra under the cup product in degrees 1
and 2. 

L. Grimley, V. C. Nguyen, and S. Witherspoon \cite{GNW} calculated Gerstenhaber brackets on Hochschild cohomology of a twisted tensor product of algebras. S. Sanchez-Flores \cite{SF} also calculated the bracket on group algebras of a cyclic group over a field of positive characteristic which is isomorphic to $A=k[x]/(x^p)$. C. Negron and S. Witherspoon \cite{NW} calculated the bracket on group algebras of a cyclic group over a field of positive characteristic as well with the same $h,\phi$, and $\Delta$ maps. Our calculation agrees with those except slightly different $[(x^i\xi_1^{*},x^j\xi_2^{*})]$.
\section{Bracket on Hopf algebra cohomology of a Taft algebra}
The Taft algebra $T_p$ with $p>2$ is a $k$-algebra generated by $g$ and $x$ satisfying the relations : $g^p = 1, x^p = 0,\text{ and }xg =\omega gx$ where $\omega$ is a primitive $p$-th root of unity. It is a Hopf algebra with the structure:
\begin{itemize}
	\item $\Delta(g)=g\otimes g$, $\Delta(x)=1\otimes x+x\otimes g$
	\item $\varepsilon(g)=1,\varepsilon(x)=0$
	\item $S(g)=g^{-1},S(x)=-xg^{-1}.$
\end{itemize}
Note that as an algebra, $T_p$ is a skew group algebra $A\rtimes kG$ where $A=k[x]/(x^p)$ and $G=<g\mid g^p=1>$. The action of $G$ on $A$ is given by $^gx=\omega x$.

In this section, our main goal is to calculate the bracket on Hochschild cohomology of $T_p$ with the same technique in Section 3 and find the bracket on Hopf algebra cohomology of $T_p$ by using the embedding of $\text{H}^*(T_p,k)$ into $\text{HH}^*(T_p,T_p)$. 

We first find the bracket on Hochschild cohomology of $T_p$. Let $\mathcal{D}$ be the skew group algebra $A^e\rtimes G$ where the action of $G$ on $A^e$ is diagonal, i.e. $^g(a\ot b)=(^ga)\ot (^gb)$. Then, there is the following isomorphism \cite[Section 2]{BW}

 $$\mathcal{D}=A^e\rtimes G\cong \underset{g\in G}{\bigoplus}Ag\ot Ag^{-1}\subset T_p^e.$$
Hence $\mathcal{D}$ is isomorphic to a subalgebra of $T_p^e$ via $a_1\ot a_2\ot g\mapsto a_1g\ot (^{g^{-1}}a_2g^{-1})$. Moreover, $A$ is a $\mathcal{D}$-module under the following left and right action \cite[Section 4]{BW}:
$$(a_1g\ot a_2g^{-1})a_3=a_1ga_3a_2g^{-1}=a_1(^g(a_3a_2))$$
$$a_3(a_1g\ot a_2g^{-1})=a_2g^{-1}a_3a_1g=a_2(^{g^{-1}}(a_3a_1)).$$

Remember the resolution \eqref{res A}
\begin{equation*}
\mathbb{A}:\cdots \stackrel{v.}{\longrightarrow} A^{e}\stackrel{u.}{\longrightarrow} A^{e} \stackrel{v.}{\longrightarrow} A^{e} \stackrel{u.}{\longrightarrow} A^{e} \stackrel{\pi}{\longrightarrow} A \longrightarrow 0.
\end{equation*}
This is also a $\mathcal{D}$-projective resolution of A and the action of $G$ on $A^e$ is given by 

\begin{itemize}
	\item $g\cdot(a_1\ot a_2)=(^ga_1)\ot(^ga_2)$ in even degrees,
	\item $g\cdot(a_1\ot a_2)=\omega (^ga_1)\ot(^ga_2)$ in odd degrees.
\end{itemize}

From the resolution $\mathbb{A}$, we construct the following $T_p^{e}$ resolution of $T_p$: 
\begin{equation} \label{resol}
T_p^{e}\ot_{\mathcal{D}}\mathbb{A}:\cdots {\longrightarrow} T_p^{e}\ot_{\mathcal{D}}{A^{e}}{\longrightarrow} T_p^{e}\ot_{\mathcal{D}}{A^{e}} {\longrightarrow} T_p^{e}\ot_{\mathcal{D}}{A^{e}}{\longrightarrow} T_p^{e}\ot_{\mathcal{D}}{A} \longrightarrow 0.
\end{equation}
It is known that, $ T_p\cong T_p^{e}\ot_{\mathcal{D}}{A} $ as $T_p$-bimodules via the map sending $x^i\ot g^k$ to $(1\ot g^k)\ot_{\mathcal{D}}x^i$ \cite[Section 3.5]{HH}. Then we have $ A\ot T_p\cong T_p^{e}\ot_{\mathcal{D}}{A^e} $ with the $T_p$-bimodule isomorphism given by
\begin{equation}\label{iso Tp}
    \kappa(x^i\ot(x^j\ot g^k))=(1\ot g^k)\ot_{\mathcal{D}}(x^i\ot x^j).
\end{equation}

Then, we obtain the following resolution $\tilde{\mathbb{A}}$ which is isomorphic to the resolution (\ref{resol}), i.e. 
\begin{equation}\label{resol2}
\tilde{\mathbb{A}}:\cdots\stackrel{\tilde{u}.}{\longrightarrow} A\ot T_p\stackrel{\tilde{v}.} {\longrightarrow}A\ot T_p\stackrel{\tilde{u}.}{\longrightarrow} A\ot T_p \stackrel{\tilde{\pi}.}{\longrightarrow}T_p\longrightarrow 0
\end{equation}
where $\tilde{v}=v\ot id_{kG},\tilde{u}=u\ot id_{kG}$, and $\tilde{\pi}=\pi\ot id_{kG}$.

The following lemma gives us a contracting homotopy for the identity map on the resolution $\tilde{\mathbb{A}}$.
\begin{lemma}\label{tildehn}
Let $h_n$ be a contracting homotopy in \eqref{hn}. Then $\tilde{h}_n=h_n\otimes1_{kG}$ forms a contracting homotopy for the identity map on $\tilde{\mathbb{A}}$.
\end{lemma}
\begin{proof}
For $n\geq 0$, the domain of $h_n\ot 1_{kG}$ is $A\otimes A\otimes kG$ which is $A\otimes T_p$ as a vector space. Moreover, by definition of contracting homotopy, $h_n$ satisfy
$$h_{i-1}d_i+d_{i+1}h_{i}=id_{\mathbb{A}_i}.$$
Then,
\begin{align*}
\tilde{h}_{i-1}\tilde{d}_i+\tilde{d}_{i+1}\tilde{h}_{i}&=(h_{i-1}\ot id_{kG})(d_i\ot id_{kG})+(d_{i+1}\ot id_{kG})(h_{i}\ot id_{kG})\\
&=(h_{i-1}d_i\ot id_{kG})+(d_{i+1}h_{i}\ot id_{kG})=(h_{i-1}d_i+d_{i+1}h_{i})\ot id_{kG}\\
&=id_{\mathbb{A}_i}\ot id_{kG}=id_{\tilde{\mathbb{A}}_i}
\end{align*}
and that implies $\tilde{h}_n$ is a contracting homotopy for $\tilde{\mathbb{A}}$. The proof is similar for $n=-1$.
\end{proof}
 We abbreviate $a_1\ot a_2\ot g\in A\ot T_p$ by $a_1\ot a_2g$. By the Lemma \ref{tildehn}, we obtain
\begin{align*}
\tilde{h}_{-1}(x^i g)&=\xi_0x^i g,\\
\tilde{h}_{0}(x^i\xi_0x^j g)&=\sum_{l=0}^{i-1}x^l\xi_1x^{i+j-1-l} g,\\
\tilde{h}_{1}(x^i\xi_1x^j g)&=\delta_{i,p-1}x^j\xi_2 g,\\
\tilde{h}_{2n}(x^i\xi_{2n}x^jg)&=-\sum_{l=0}^{j-1}x^{i+j-1-l}\xi_{2n+1}x^lg,\\
\tilde{h}_{2n+1}(x^i\xi_{2n+1}x^jg)&=\delta_{j,p-1}x^i\xi_{2n+2} g.
\end{align*}

We need a lemma to have the linear maps $\tilde{\phi}_i:(\tilde{\mathbb{A}}\otimes_{T_p}\tilde{\mathbb{A}})_i\longrightarrow \tilde{\mathbb{A}}_{i+1}$. However, we first mention that there is an isomorphism from $(A\ot T_p)\ot_{T_p}(A\ot T_p)$ to $(A\ot A)\ot_A(A\ot A)\ot kG$ as $T_p^e$-modules given by 
\begin{equation}\label{iso ATp}
\psi((x^{i_1}\ot x^{j_1}g^{k_1})\ot_{T_p}(x^{i_2}\ot x^{j_2}g^{k_2}))=\omega^{k_1(i_2+j_2)}(x^{i_1}\ot x^{j_1})\ot_A(x^{i_2}\ot x^{j_2}) g^{(k_1+k_2)}.
\end{equation}

 \begin{lemma}\label{lemma tildephi}
 Let $F_\mathbb{A}=(\pi\otimes_A id_\mathbb{A}- id_\mathbb{A}\otimes_A \pi)$ be the chain map for the resolution $\mathbb{A}$ in \eqref{res A} which is used for calculation of $\phi$ in \eqref{phin}. Then $F_{\tilde{\mathbb{A}}}:\tilde{\mathbb{A}}\ot_{T_p}\tilde{\mathbb{A}}\to \tilde{\mathbb{A}}$ defined by $(\tilde{\pi}\otimes_{T_p} id_{\tilde{\mathbb{A}}}- id_{\tilde{\mathbb{A}}}\otimes_{T_p} \tilde{\pi})$ is exactly $(F_\mathbb{A}\ot id_{kG})\psi$. Moreover $\tilde{\phi}:=(\phi\ot id_{kG})\psi$ is a contracting homotopy for $F_{\tilde{\mathbb{A}}}$.
 \end{lemma}
 \begin{proof}
 Let $(x^{i_1}\ot x^{j_1}g^{k_1})\ot_{T_p}(x^{i_2}\ot x^{j_2}g^{k_2})\in (A\ot T_p)\ot_{T_p}(A\ot T_p)$. Note that $F_{\tilde{\mathbb{A}}}$ is zero if degrees of $(x^{i_1}\ot x^{j_1}g^{k_1})$ and $(x^{i_2}\ot x^{j_2}g^{k_2})$ are both nonzero since $\tilde{\pi}$ is only defined on degree zero. Also remember that $\tilde{\pi}=\pi\ot id_{kG}$ for the resolution $\tilde{\mathbb{A}}$.
 
We check the case that the degree of $(x^{i_1}\ot x^{j_1}g^{k_1})$ is zero and the degree of $(x^{i_2}\ot x^{j_2}g^{k_2})$ is nonzero. By using definition of $F_{\tilde{\mathbb{A}}}$, we obtain
 \begin{align*}
F_{\tilde{\mathbb{A}}}((x^{i_1}\ot x^{j_1}g^{k_1})\ot_{T_p}(x^{i_2}\ot x^{j_2}g^{k_2}))&=(x^{i_1+j_1}g^{k_1})\ot_{T_p}(x^{i_2}\ot x^{j_2}g^{k_2})\\
&=\omega^{k_1(i_2+j_2)}x^{i_1+i_2+j_1}\ot x^{i_2}g^{k_1+k_2}.
 \end{align*}
 
 On the other hand, we also have 
  \begin{align*}
&(F_\mathbb{A}\ot id_{kG})\psi((x^{i_1}\ot x^{j_1}g^{k_1})\ot_{T_p}(x^{i_2}\ot x^{j_2}g^{k_2}))\\
&=(F_\mathbb{A}\ot id_{kG})(\omega^{k_1(i_2+j_2)}(x^{i_1}\ot x^{j_1})\ot_A (x^{i_2}\ot x^{j_2})g^{k_1+k_2})\\
&=\omega^{k_1(i_2+j_2)}x^{i_1+i_2+j_1}\ot x^{i_2}g^{k_1+k_2}.
 \end{align*}
 The proof for other cases are similar. Hence $F_{\tilde{\mathbb{A}}}$ and $(F_\mathbb{A}\ot id_{kG})\psi$ are identical.
 
 In order to prove $\tilde{\phi}:=(\phi \ot id_{kG})\psi$ is a contracting homotopy for $F_{\tilde{\mathbb{A}}}$, we need to show that $$\tilde{d}_{\tilde{\mathbb{A}}}\tilde{\phi}+\tilde{\phi}\tilde{d}_{\tilde{\mathbb{A}}\ot_{T_p}\tilde{\mathbb{A}}}= F_{\tilde{\mathbb{A}}}.$$
 It is clear that
 \begin{equation}\label{d phi}
    \tilde{d}_{\tilde{\mathbb{A}}}\tilde{\phi}=(d_{\mathbb{A}}\ot id_{kG})(\phi \ot id_{kG})\psi=(d_{\mathbb{A}}\phi \ot id_{kG})\psi.
 \end{equation}

We now claim that 
\begin{equation}\label{d_AotA}
\psi\tilde{d}_{\tilde{\mathbb{A}}\ot_{T_p}\tilde{\mathbb{A}}}=(d_{\mathbb{A}\ot_A \mathbb{A}}\ot id_{kG})\psi.
\end{equation}
By definition $$\tilde{d}_{\tilde{\mathbb{A}}\ot_{T_p}\tilde{\mathbb{A}}}=\tilde{d}_{\tilde{\mathbb{A}}}\ot_{T_p}id_{T_p}+(-1)^{*} id_{T_p} \ot_{T_p} \tilde{d}_{\tilde{\mathbb{A}}}$$ where $*$ is the degree of the element in left $A\ot T_p$. Moreover, $(A\ot T_p)\ot_{T_p}(A\ot T_p)$ is generated by $\xi_m1_G\ot_{T_p} x^i\xi_n1_G$ as $T_p$-bimodule. Without loss of generality, assume $m$ and $n$ are odd. Then we have the following calculation:
 \begin{flalign*}
&\begin{aligned}
&\psi \tilde{d}_{\tilde{\mathbb{A}}\ot_{T_p}\tilde{\mathbb{A}}}(\xi_m1_G\ot_{T_p} x^i\xi_n1_G)\\
&=\psi ((x\xi_m1_G-\xi_mx1_G)\ot_{T_p}x^{i}\xi_n1_G- \xi_m1_G\ot_{T_p} (x^{i+1}\xi_n1_G-x^i\xi_n x1_G))\\
&=(x\xi_m-\xi_mx)\ot_Ax^{i}\xi_n1_G-\xi_m\ot_A (x^{i+1}\xi_n-x^i\xi_n x)1_G
 \end{aligned}&&
\end{flalign*}
 and 
 \begin{flalign*}
&\begin{aligned}
&(d_{\mathbb{A}\ot_A \mathbb{A}}\ot id_{kG})\psi(\xi_m1_G\ot_{T_p} x^i\xi_n1_G)\\
&=(d_{\mathbb{A}\ot_A \mathbb{A}}\ot id_{kG})(\xi_m\ot_A x^i\xi_n1_G)\\
&=(x\xi_m-\xi_mx)\ot_Ax^{i}\xi_n1_G-\xi_m\ot_A (x^{i+1}\xi_n-x^i\xi_n x)1_G.
\end{aligned}&&
\end{flalign*}
The calculation is similar for the other cases of $m$ and $n$. Therefore,
\begin{equation}\label{phi d}
\tilde{\phi}\tilde{d}_{\tilde{\mathbb{A}}\ot_{T_p}\tilde{\mathbb{A}}}=(\phi \ot id_{kG})\psi\tilde{d}_{\tilde{\mathbb{A}}\ot_{T_p}\tilde{\mathbb{A}}}=(\phi \ot id_{kG})(d_{\mathbb{A}\ot_A\mathbb{A}} \ot id_{kG})\psi=(\phi d_{\mathbb{A}\ot_A\mathbb{A}} \ot id_{kG})\psi.    
\end{equation}
By combining \eqref{d phi} and \eqref{phi d}, we obtain
\begin{equation*}
  \tilde{d}_{\tilde{\mathbb{A}}}\tilde{\phi}+\tilde{\phi}\tilde{d}_{\tilde{\mathbb{A}}\ot_{T_p}\tilde{\mathbb{A}}}=((d_\mathbb{A}\phi+\phi d_{ \mathbb{A}\otimes_A \mathbb{A}})\ot id_{kG})\psi=(F_\mathbb{A}\ot id_{kG})\psi= F_{\tilde{\mathbb{A}}}  
\end{equation*}
whence $\tilde{\phi}=(\phi\ot id_{kG})\psi$ is a contracting homotopy for $F_{\tilde{\mathbb{A}}}$.
 \end{proof}
 We use the Lemma \ref{lemma tildephi} and find the following $T_p^e$-linear maps $\tilde{\phi}_i:(\tilde{\mathbb{A}}\otimes_{T_p}\tilde{\mathbb{A}})_i\longrightarrow \tilde{\mathbb{A}}_{i+1}$:
\begin{align*}
&\tilde{\phi}_0(\xi_01_G\ot_{T_p} x^i\xi_01_G)=\sum_{l=0}^{i-1}x^l\xi_1x^{i-1-l}1_G,\\
&\tilde{\phi}_1(\xi_11_G\ot_{T_p} x^i\xi_01_G)=-\delta_{i,p-1}\xi_21_G,\\
&\tilde{\phi}_1(\xi_01_G\ot_{T_p} x^i\xi_11_G)=\delta_{i,p-1}\xi_21_G.
\end{align*}

Next, we give a lemma to find the the diagonal map.
\begin{lemma}\label{lemma diagonal}
The map $\tilde{\Delta}:=\psi^{-1}(\Delta\ot id_{kG})$ is a diagonal map on $\tilde{\mathbb{A}}$ where $\Delta$ is in \eqref{Deltan}.
\end{lemma}
\begin{proof}
We need to check that $\tilde{\Delta}$ is a chain map. The following equations are straightforward by considering the fact that $\Delta$ is a chain map and \eqref{d_AotA}:
\begin{align*}
\tilde{d}_{\tilde{\mathbb{A}}\ot_{T_p}\tilde{\mathbb{A}}}\tilde{\Delta}&=\tilde{d}_{\tilde{\mathbb{A}}\ot_{T_p}\tilde{\mathbb{A}}}\psi^{-1}(\Delta\ot id_{kG})=\psi^{-1}(d_{\mathbb{A}\ot_A \mathbb{A}}\ot id_{kG})(\Delta\ot id_{kG})\\
&=\psi^{-1}(d_{\mathbb{A}\ot_A \mathbb{A}}\Delta\ot id_{kG})=\psi^{-1}(\Delta d_{\mathbb{A}}\ot id_{kG})=\psi^{-1}(\Delta\ot id_{kG})( d_{\mathbb{A}}\ot id_{kG})\\
&=\tilde{\Delta}\tilde{d}_{\tilde{\mathbb{A}}}.
\end{align*}
\end{proof}
Lemma \ref{lemma diagonal} allows us to compute the $T_p$-linear map $\tilde{\Delta}:\tilde{\mathbb{A}}_{i+1}\longrightarrow (\tilde{\mathbb{A}}\otimes_{T_p}\tilde{\mathbb{A}})_i$ as follows:
\begin{align*}
\tilde{\Delta}_0(\xi_0 1_G) &= \xi_01_G\ot_{T_p}\xi_01_G,\\
\tilde{\Delta}_1(\xi_11_G) &=\xi_11_G\ot_{T_p}\xi_0 1_G + \xi_01_G \ot_{T_p} \xi_11_G,\\
\tilde{\Delta}_{2n}(\xi_{2n}1_G) &= \sum_{i=0}^{n}\xi_{2i} 1_G\ot_{T_p}\xi_{2n-2i} 1_G \\
&+\sum_{i=0}^{n-1}\sum_{\substack{a+b+c \\ =p-2}}^{}x^a\xi_{2i+1} 1_G\ot_{T_p}x^b\xi_{2n-2i-1}x^c1_G, \text{ for }n\geq 1\\
\tilde{\Delta}_{2n+1}(\xi_{2n+1} 1_G) &= \sum_{i=0}^{2n+1}\xi_i 1_G\ot_{T_p} \xi_{2n+1-i} 1_G, \text{ for }n\geq 1.
\end{align*}

Before computing the bracket on Hochschild cohomology of $T_p$, we need to find a basis of $\text{Hom}_{T_p^e}(\tilde{\mathbb{A}}, T_p)$. In particular, we must find a basis of $\text{Hom}_{T_p^e}(A\ot T_p, T_p)$ as it is an invariant in each degree. 

It is known that $$\HH^*(T_p):=\text{Ext}^*_{T_p^{e}}(T_p,T_p)\cong\text{Ext}^*_{\mathcal{D}}(A,T_p)\cong \text{Ext}^*_{A^{e}}(A,T_p)^G.$$ 
The Eckmann-Shapiro Lemma (Lemma \ref{Eck}) and (\ref{iso Tp}) imply the first isomorphism and see \cite[Theorem 3.6.2]{HH} for the second isomorphism. 

Consider the following resolution 
\begin{equation}\label{2nd}
\text{Hom}_{A^e}(\mathbb{A},T_p)^G:0{\longrightarrow} \text{Hom}_{A^e}(A^e,T_p)^G{\longrightarrow} \text{Hom}_{A^e}(A^e,T_p)^G {\longrightarrow} \cdots
\end{equation}
where the action of $G$ on $\text{Hom}_{A^e}(A^e,T_p)^G$ is defined by 
\begin{equation}\label{action on T_p}
g\cdot f(a_1\ot a_2)={^g}f(^{g^{-1}}(a_1\ot a_2)).
\end{equation} 
This resolution is clearly isomorphic to 
\begin{equation}\label{1st}
0{\longrightarrow} T_p^{G}{\longrightarrow} T_p^{G} {\longrightarrow} T_p^{G} \longrightarrow \cdots
\end{equation}
with the correspondence 
\begin{equation}\label{corres}
f_t\mapsto t \text{ where } f_t(\xi_*)=t\text{ for all } t\in T_p.
\end{equation}
 We claim that $\text{Hom}_{T_p^e}(A\ot T_p, T_p)\cong T_p^G$. Suppose $x^ig^j \in T_p^{G}$. Then, we have $f_{x^ig^j}\in \text{Hom}_{A^e}(A^e,T_p)^G$ defined by $f_{x^ig^j}(x^k\ot x^l):=x^{k+l+i}g^j$ where $x^*\in A$. Now observe that, $f_{x^ig^j}\in \text{Hom}_{A^e}(A^e,T_p)^G$ is a $\mathcal{D}$-module homomorphism since
\begin{align*}
f_{x^ig^j}((x^k\xi_* x^l g)(a_1\ot a_2))&=f_{x^ig^j}((x^k\xi_* x^l 1_G)g(a_1\ot a_2))=(x^k\xi_* x^l 1_G)f_{x^ig^j}(g(a_1\ot a_2))\\
&=(x^k\xi_*x^l 1_G)gf_{x^ig^j}(a_1\ot a_2)=(x^k\xi_* x^l g)f_{x^ig^j}(a_1\ot a_2)
\end{align*}
where $x^k\xi_* x^l g\in \mathcal{D}, a_1\ot a_2\in A^e$. Moreover, if $f\in \text{Hom}_{\mathcal{D}}(A^e,T_p)$, then $f$ is $G$-invariant as 
$$g\cdot f(a_1\ot a_2)={^g}f(^{g^{-1}}(a_1\ot a_2))={^{(gg^{-1})}}f(a_1\ot a_2)=f(a_1\ot a_2)$$
where $g\in G, a_1\ot a_2\in A^e$. Hence, the isomorphism 
from $\text{Hom} _{A^e}(A^e,T_p)^G$ to $\text{Hom}_{\mathcal{D}}(A^e,T_p)$ is the identity, so that $f_{x^ig^j}$ is also in $\text{Hom}_{\mathcal{D}}(A^e,T_p)$. We next use the Eckmann-Shapiro lemma (Lemma \ref{Eck}) which implies that $\text{Ext}_{\mathcal{D}}^{*}(A,T_p)\cong \text{Ext}_{T_p^{e}}^{*}(T_p^{e}\ot_\mathcal{D}A,T_p)$ and the isomorphism is given by 
\begin{align*}
\sigma(f_{x^ig^j})(x^mg^s\ot x^ng^r\ot_{\mathcal{D}}x^k\ot x^l)&=x^mg^s\ot x^ng^rf_{x^ig^j}(x^k\ot x^l)=x^mg^s\ot x^ng^r(x^{k+l+i}g^j)\\
&=(x^mg^s)(x^{k+l+i}g^j)(x^ng^r)\\
&=((x^m(^{g^s}x^{k+l+i}))g^{s+j})(x^ng^r)\\
&=\omega^{s(k+l+i)}(x^{m+k+l+i}g^{s+j})(x^ng^r)\\
&=\omega^{s(k+l+i)}(x^{m+k+l+i}(^{g^{s+j}}x^n))g^{j+s+r}\\
&=\omega^{s(k+l+i+n)+jn}x^{i+k+l+m+n}g^{j+s+r}.
\end{align*}
Hence, $\sigma(f_{x^ig^j})$ is in $\text{Hom}_{T_p^{e}}(T_p^{e}\ot_{\mathcal{D}} A^e,T_p)$. Lastly, recall that $T_p^{e}\ot_{\mathcal{D}}A^e\cong A\ot T_p$ via $\kappa$ (\ref{iso Tp}); so that,
\begin{align*}
\kappa^*(\sigma(f_{x^ig^j}))(x^k\ot x^lg^r)=\sigma(f_{x^ig^j})((1_{T_p}\ot \xi_*g^r)\ot_{\mathcal{D}}x^k\ot x^l)=x^{i+k+l}g^{j+r}
\end{align*}
which implies $\kappa^{*}(\sigma(f_{x^ig^j}))\in \text{Hom}_{T_p^e}(A\ot T_p, T_p)$. For simplicity, we define $\tilde{f}_{x^ig^j}:=\kappa^{*}(\sigma(f_{x^ig^j}))$.

The action of $G$ on $T_p$ given by \eqref{action on T_p} and \eqref{corres} depends on degree. Since $T_p^{G}$ is spanned by $\{1,g,\cdots,g^{p-1}\}$ in even degrees and $\{x,xg,\cdots, xg^{p-1}\}$ in odd degrees \cite[Section 8.2]{Ng}, we have $\{\tilde{f}_1,\tilde{f}_g,\cdots, \tilde{f}_{g^{p-1}}\}$ in even degrees and $\{\tilde{f}_x,\tilde{f}_{xg},\cdots,\tilde{f}_{xg^{p-1}}\}$ in the odd degrees as a basis of $\text{Hom}_{T_p^e}(A\ot T_p, T_p)$.

We only calculate the bracket in degree 1 and 2 as before so we can extend it to higher degrees by the relation between cup product and the bracket. Since $A\ot T_p\cong A^e\ot kG$ as vector spaces, $\xi_i 1_G$ generates $A\ot T_p$ as a $T_p$-bimodule. Through the calculation, $id$ represents $id_{A\ot T_p}$ and $\ot$ represents $\ot_{T_p}$.

The circle product of two elements in degree one is
\begin{align*}
(\tilde{f}_{xg^i}\circ_{\tilde{\phi}}\tilde{f}_{xg^j})(\xi_11_G)&=\tilde{f}_{xg^i}\tilde{\phi}(id\otimes\tilde{f}_{xg^j}\ot id)\tilde{\Delta}^{(2)}(\xi_1 1_G)\\
&=\tilde{f}_{xg^i}\tilde{\phi}(id\ot \tilde{f}_{xg^j}\ot id)(\xi_01_G\ot\xi_01_G\ot\xi_11_G+\xi_01_G\ot\xi_11_G\ot\xi_01_G\\
&+\xi_11_G\ot\xi_01_G\ot\xi_01_G)\\
&=\tilde{f}_{xg^i}\tilde{\phi}(\xi_01_G\ot  x\xi_0g^j)=\tilde{f}_{xg^i}(\xi_1 g^j)=xg^{i+j}.
\end{align*}
Because of the symmetry, $(\tilde{f}_{xg^j}\circ_{\tilde{\phi}}\tilde{f}_{xg^i})(\xi_11_G)=xg^{i+j}$. Therefore $$[\tilde{f}_{xg^i},\tilde{f}_{xg^j}](\xi_11_G)=xg^{i+j}-(-1)^{0} xg^{i+j}=0.$$
The circle product of the elements of degrees 1 and 2:
\begin{align*}
(\tilde{f}_{xg^i}\circ_{\tilde{\phi}}\tilde{f}_{g^j})(\xi_21_G)&=\tilde{f}_{xg^i}\tilde{\phi}(id\ot\tilde{f}_{g^j}\ot id)\tilde{\Delta}^{(2)}(\xi_2 1_G)=\tilde{f}_{xg^i}\tilde{\phi}(id\ot \tilde{f}_{g^j}\ot id)\\
&(\xi_01_G\ot\xi_01_G\ot\xi_21_G+\xi_01_G\ot\xi_21_G\ot\xi_01_G\\
&+\xi_01_G\ot\sum_{\substack{a+b+c \\ =p-2}}(x^a\xi_11_G\ot x^b\xi_1x^c1_G)+\xi_21_G\ot\xi_01_G\ot\xi_01_G\\
&+\sum_{\substack{a+b+c \\ =p-2}}(x^a\xi_11_G\ot (x^b\xi_01_G\ot \xi_1x^c1_G+x^b\xi_11_G\ot\xi_0x^c1_G)))\\
&=\tilde{f}_{xg^i}\tilde{\phi}(\xi_01_G\ot\xi_0g^j)=0.
\end{align*}
And the circle product on the reverse order:
\begin{align*}
(\tilde{f}_{g^j}\circ_{\tilde{\phi}}\tilde{f}_{xg^i})(\xi_21_G)&=\tilde{f}_{g^j}\tilde{\phi}(id\ot\tilde{f}_{xg^i}\ot id)\tilde{\Delta}^{(2)}(\xi_2 1_G)=\tilde{f}_{g^j}\tilde{\phi}(id\ot \tilde{f}_{xg^i}\ot id)\\
&(\xi_01_G\ot\xi_01_G\ot\xi_21_G+\xi_01_G\ot\xi_21_G\ot\xi_01_G\\
&+\xi_01_G\ot\sum_{\substack{a+b+c \\ =p-2}}(x^a\xi_11_G\ot x^b\xi_1x^c1_G)+\xi_21_G\ot\xi_01_G\ot\xi_01_G\\
&+\sum_{\substack{a+b+c \\ =p-2}}(x^a\xi_11_G\ot (x^b\xi_01_G\ot \xi_1x^c1_G+x^b\xi_11_G\ot \xi_0x^c1_G)))\\
&=\tilde{f}_{g^j}\tilde{\phi}(\sum_{\substack{a+b+c \\ =p-2}}\omega^{i(b+c)}\xi_01_G\ot x^{a+b+1}\xi_1x^cg^i+\omega^{ic}x^a\xi_11_G\ot x^{b+1}\xi_0x^cg^i)\\
&=\tilde{f}_{g^j}(\sum_{\substack{a+b+c \\ =p-2}}\omega^{i(b+c)}\delta_{a+b+1,p-1}x^c\xi_2g^i-\omega^{ic}\delta_{b+1,p-1}x^{a+c}\xi_2 g^i)\\
&=\tilde{f}_{g^j}(\sum_{b=0}^{p-2}\omega^{ib}\xi_2g^i)-\tilde{f}_{g^j}(\xi_2g^i)\\
&=\left\{\begin{array}{lr}
        (p-2)g^j, & \text{for } i=0\\
        -(\omega^{-i}+1)g^{i+j}, & \text{for } i\neq 0
        \end{array}\right..
\end{align*}
Therefore, we obtain
$$[\tilde{f}_{xg^i},\tilde{f}_{g^j}]=\left\{\begin{array}{lr}
        -(p-2)g^j, & \text{for } i=0\\
        (\omega^{-i}+1)g^{i+j}, & \text{for } i\neq 0
        \end{array}\right..$$
Lastly, the bracket of the elements of degrees 2 and 2:\\
\begin{align*}
(\tilde{f}_{g^i}\circ_{\tilde{\phi}}\tilde{f}_{g^j})(\xi_31_G)&=\tilde{f}_{g^i}\tilde{\phi}(id\ot\tilde{f}_{g^j}\ot id)\tilde{\Delta}^{(2)}(\xi_31_G)=\tilde{f}_{g^i}\tilde{\phi}(id\ot\tilde{f}_{g^j}\ot id)\\
&(\xi_01_G\ot\xi_01_G\ot\xi_31_G+\xi_01_G\ot\xi_11_G\ot\xi_21_G+\xi_01_G\ot\xi_21_G\ot\xi_11_G\\
&+\xi_01_G\ot\xi_31_G\ot\xi_01_G+\xi_11_G\ot\xi_21_G\ot\xi_01_G+\xi_11_G\ot\xi_01_G\ot\xi_21_G\\
&+\xi_21_G\ot\xi_11_G\ot\xi_01_G+\xi_21_G\ot\xi_01_G\ot\xi_11_G+\xi_31_G\ot\xi_01_G\ot\xi_01_G)\\
&=\tilde{f}_{g^i}\tilde{\phi}(\xi_01_G\ot\xi_1g^j+\xi_11_G\ot\xi_0g^j)=0
\end{align*}
and by symmetry
$(\tilde{f}_{g^j}\circ_{\tilde{\phi}}\tilde{f}_{g^i})(\xi_31_G)=0$. Therefore, we have 
$[\tilde{f}_{g^i},\tilde{f}_{g^j}]=0$.
As a consequence, the bracket for the elements of degree 1 and 2 are 
$$[\tilde{f}_{xg^i},\tilde{f}_{xg^j}]=0,[\tilde{f}_{xg^i},\tilde{f}_{g^j}]=\left\{\begin{array}{lr}
        -(p-2)g^j, & \text{for } i=0\\
        (\omega^{-i}+1)g^{i+j}, & \text{for } i\neq 0
        \end{array}\right.,[\tilde{f}_{g^i},\tilde{f}_{g^j}]=0.$$
By the identity (\ref{cupbrac}), brackets in higher degrees can be determined, since the
Hochschild cohomology is generated as an algebra under cup product in degrees 1
and 2.

Hopf algebra cohomology of $T_p$ and Hochschild cohomology of $T_p$ were calculated before by V. C. Nguyen \cite[Section 8]{Ng} as the Hopf algebra cohomology
\[
\text{H}^n(T_p,k)=
\begin{cases}
k & \text{if $n$ is even,} \\
0 & \text{if $n$ is odd,} 
\end{cases}
\]
and the Hochschild cohomology
\[
\text{HH}^n(T_p,k)=
\begin{cases}
k & \text{if $n$ is even,} \\
Span_k\{x\} & \text{if $n$ is odd.}
\end{cases}
\]

It is known that for any Hopf algebra with bijective antipode, the Hopf algebra cohomology can be embedded into the Hochschild cohomology \cite[Theorem 9.4.5 and Corollary 9.4.7]{HH}. Since any finite dimensional Hopf algebra has a bijective antipode, the Taft algebra $T_p$ is also a Hopf algebra with a bijective antipode. The embedding of $\text{H}^n(T_p,k)$ into $\text{HH}^n(T_p,T_p)$ turns out to be the map that is identity in even degrees and zero on odd degrees. Then, the corresponding bracket in Hopf algebra cohomology is
$$[\tilde{f}_{g^i},\tilde{f}_{g^j}]=0,$$
so that, the bracket on Hopf algebra cohomology for the elements of all degrees is 0 by the identity (\ref{cupbrac}). 

This is the first example of the Gerstenhaber bracket on the Hopf algebra cohomology of a nonquasi-triangular Hopf algebra and our calculation shows that the bracket on Hopf algebra cohomology of a Taft algebra is zero as it is on the Hopf algebra cohomology of any quasi-triangular algebra. A natural question that arises whether the bracket structure on the Hopf algebra cohomology is always trivial. In the next section, we explore a general expression for the bracket on the Hopf algebra cohomology that may help us to approach this question with a more theoretical perspective in the future researches.

\section{Gerstenhaber bracket for Hopf algebras}

In this section, we want to explore an expression for Gerstenhaber bracket on a Hopf algebra $A$ with a bijective antipode $S$. 

We give the following lemma which helps us to define the Gerstenhaber bracket on an equivalent resolution to the bar resolution of $A$ as an $A$-bimodule.
 \begin{lemma}
 	Let $A$ be a Hopf algebra with bijective antipode. Let $P_{\bullet}$ be the bar resolution of $k$ as a left $A$-module:
 	\begin{equation*}
 	P_{\bullet}:\cdots \stackrel{d_3}{\longrightarrow} A^{\otimes 3} \stackrel{d_2}{\longrightarrow} A^{\otimes 2} \stackrel{d_1}{\longrightarrow} A \stackrel{\varepsilon}\longrightarrow k\longrightarrow 0,
 	\end{equation*}
 	with differentials $$d_n(a_0\otimes a_1\otimes\cdots \otimes a_{n})=\sum_{i=0}^{n-1}(-1)^ia_0\otimes a_1\otimes \cdots \otimes a_ia_{i+1}\otimes\cdots \otimes a_n+(-1)^n\varepsilon(a_n)a_0\ot \cdots\ot a_{n-1}$$
 	Then $X_{\bullet}=A^e\ot_AP_{\bullet}$ is equivalent to the bar resolution of $A$ as an  $A$-bimodule.
 \end{lemma}
 \begin{proof}
Since $S$ is bijective \cite[Lemma 9.2.9]{HH}, $A^e$ is projective as a right $A$-module. Also there is an $A^e$-module isomorphism $\rho:A\to A^e\ot_Ak$ defined by $\rho(a)=a\ot 1\ot 1$ for all $a\in A$ \cite[Lemma 9.4.2]{HH}. 

For each $n$, define $\theta_n:X_n\to A^{\ot (n+2)}$ by
$$\theta_n((a\ot b)\ot_A(1\ot c^1\ot c^2\ot \cdots \ot c^n))=\sum a\ot c_1^1\ot c_1^2\ot\cdots\ot c_1^n\ot S(c_2^1c_2^2\cdots c_2^n)b$$
for all $a,b,c^1,\cdots c^n\in A.$

Now, we show that $\theta$ is a chain map: 
\begin{flalign*}
\theta&_{n-1}d_n((a\ot b)\ot_A(1\ot c^1\ot c^2\ot \cdots \ot c^n))\\
=&\theta_{n-1}((a\ot b)\ot_A(c^1\ot c^2\ot \cdots \ot c^n)\\
&+\sum_{i=1}^{n-1}(-1)^i(a\ot b)\ot_A(1\ot c^1\ot c^2\ot \cdots\ot c^ic^{i+1}\ot\cdots  \ot c^n)\\
&+(-1)^n(a\ot b)\ot_A(\varepsilon(c^n)\ot c^1\ot c^2\ot \cdots \ot c^{n-1}))\\
=&\theta_{n-1}(\sum (ac_1^1\ot S(c_2^1) b)\ot_A(1\ot c^2\ot \cdots \ot c^n)\\
&+\sum_{i=1}^{n-1}(-1)^i(a\ot b)\ot_A(1\ot c^1\ot c^2\ot \cdots\ot c^ic^{i+1}\ot\cdots  \ot c^n)\\
&+(-1)^n(\varepsilon(c^n)a\ot b)\ot_A(1 \ot c^1\ot c^2\ot \cdots \ot c^{n-1}))\\
=&\sum ac_1^1\ot c_1^2\ot \cdots \ot c_1^n\ot S(c_2^2\cdots c_2^n) S(c_2^1)b\\
&+\sum_{i=1}^{n-1}(-1)^i\sum a\ot c_1^1\ot \cdots \ot c_1^ic_1^{i+1}\ot \cdots \ot c_1^n\ot S(c_2^1\cdots c_2^n)b\\
&+\sum (-1)^na\ot c_1^1\ot \cdots \ot c_1^{n-1}\ot \varepsilon(c^n)S(c_2^1\cdots c_2^{n-1})b
\end{flalign*}
and
\begin{align*}
d_n&\theta_n((a\ot b)\ot_A(1\ot c^1\ot c^2\ot \cdots \ot c^n))\\
=&d_n(\sum a\ot c_1^1\ot c_1^2\ot\cdots\ot c_1^n\ot S(c_2^1c_2^2\cdots c_2^n)b)\\
=&\sum ac_1^1\ot c_1^2\ot \cdots \ot c_1^n\ot S(c_2^1c_2^2\cdots c_2^n)b\\
&+\sum\sum_{i=1}^{n-1}(-1)^i a\ot c_1^1\ot \cdots \ot c_1^ic_1^{i+1}\ot \cdots \ot c_1^n\ot S(c_2^1\cdots c_2^n)b\\
&+\sum (-1)^na\ot c_1^1\ot \cdots \ot c_1^{n-1}\ot c_1^nS(c_2^1\cdots c_2^{n})b.
\end{align*}

Since $S$ is an algebra anti-homomorphism that is convolution inverse to the identity map,
$$\sum c_1^nS(c_2^1\cdots c_2^n)=\sum c_1^nS(c_2^n)S(c_2^{n-1})\cdots S(c_2^1)=\sum \varepsilon(c^n)S(c_2^1\cdots c_2^{n-1})$$
and
$$S(c_2^2\cdots c_2^{n})S(c_2^1)=S(c_2^1c_2^2\cdots c_2^{n})$$
so that the two expressions are equal which follows $\theta$ is a chain map.

Lastly, one can see that the $A^e$-module homomorphism $$\psi_n(a\ot c^1\ot c^2\ot \cdots c^n\ot b)=\sum (a\ot c_2^1c_2^2\cdots c_2^nb)\ot_A (1\ot c_1^1\ot c_1^2\ot \cdots \ot c_1^n)$$
is the inverse of $\theta_n$ by using the property that $S$ is an algebra anti-homomorphism that is convolution inverse to the identity map.
 \end{proof}
 
  Let $f_x\in$Hom$_{A^e}(X_m,A)$ and $g_x\in$Hom$_{A^e}(X_n,A)$. Then we define the $X$-bracket $[f_x,g_x]_X\in$Hom$_{A^e}(X_{m+n-1},A)$ to be a composition $X\stackrel{\theta}\longrightarrow B(A)\xrightarrow{[\psi^{*}f_x.\psi^{*}g_x]} A$; so that, we have $$[f_x,g_x]_X=[\psi^{*}f_x,\psi^{*}g_x]\theta=(\psi^{*}f_x\circ \psi^{*}g_x)\theta-(-1)^{(m-1)(n-1)}(\psi^{*}g_x\circ \psi^{*}f_x)\theta$$
 where
 \begin{align*}
 (\psi&^{*}f_x\circ \psi^{*}g_x)\theta_{m+n-1}((a\ot b)\ot_A 1\ot c^1\otimes \cdots \otimes c^{m+n-1})\\ 
 =&(\psi^{*}f_x\circ \psi^{*}g_x)(\sum a\ot c_1^1\ot c_1^2\ot \cdots \ot c_1^{m+n-1}\ot S(c_2^1c_2^2\cdots c_2^{m+n-1})b)\\
 =&\sum \sum_{i=1}^{m}(-1)^{(n-1)(i-1)}f_x\psi_m(a\otimes c_1^1\otimes\cdots \ot  c_1^{i-1}\otimes g_x\psi_n(1\ot c_1^{i}\otimes \cdots \otimes c_1^{i+n-1}\ot 1)\\
 &\otimes c_1^{i+n}\otimes \cdots \otimes c_1^{m+n-1}\ot S(c_2^1c_2^2\cdots c_2^{m+n-1})b)\\
 =&\sum \sum_{i=1}^{m}(-1)^{(n-1)(i-1)}f_x\psi_m(a\otimes c_1^1\otimes\cdots \ot  c_1^{i-1}\\
 & \otimes\sum g_x(1\ot c_{2}^{i}c_2^{i+1}\cdots c_{2}^{i+n-1}\ot_A 1\ot c_{1}^{i}\ot c_{1}^{i+1}\ot \cdots \ot c_{1}^{i+n-1} )\\
 &\otimes c_1^{i+n}\otimes \cdots \otimes c_1^{m+n-1}\ot S(c_2^1\cdots c_2^{i-1}c_3^{i}\cdots c_3^{i+n-1}c_2^{i+n} \cdots  c_2^{m+n-1})b)\\
 =&\sum \sum_{i=1}^{m}(-1)^{(n-1)(i-1)}f_x(a\otimes  c_{2}^{1}c_{2}^{2}\cdots c_{2}^{i-1}c_{2}^{*} c_{2}^{i+n}\cdots c_{2}^{m+n-1}S(c_3^1c_3^2\cdots c_3^{m+n-1})b\\
 & \ot_A 1\ot c_{1}^{1}\ot c_{1}^{2}\ot \cdots \ot c_{1}^{i-1}\ot c_{1}^{*}\ot  c_{1}^{i+n}\ot \cdots \ot c_{1}^{m+n-1})\\
 \end{align*}
 
 where 
 \begin{align*}
 \Delta(c)&=\sum c_1\ot c_2, \Delta^{(2)}(c)=\sum c_1\ot c_2\ot c_3,\Delta(c^*)=\sum c_{1}^*\ot c_{2}^*\text{ and }\\
 c^*&=\sum g_x(1\ot c_{2}^{i}c_2^{i+1}\cdots c_{2}^{i+n-1}\ot_A 1\ot c_{1}^{i}\ot c_{1}^{i+1}\ot \cdots \ot c_{1}^{i+n-1} ).
 \end{align*}
 
 This is the general expression of the Gerstenhaber bracket on Hochschild cohomology of $A$. Next, we start with the following theorem \cite[Theorem 9.4.5]{HH} to construct an embedding from H$^*(A,k)$ into HH$^*(A)$.
 \begin{theo}
 Let $A$ be a Hopf algebra over k with bijective antipode. Then $$\HH^*(A)\cong \coh^*(A,A^{ad}).$$
 \end{theo}
  In this theorem $A^{ad}$ is an $A$-module $A$ under left adjoint action, given by \\
  $a\cdot b=\sum a_1bS(a_2)$ for all $a,b\in A$. To find explicit isomorphism between HH$^*(A)$ and H$^*(A,A^{ad})$, we give the Eckmann-Shapiro lemma.
 \begin{lemma}[Eckmann-Shapiro]\label{Eck}
 Let $A$ be a ring and let $B$ be a subring of $A$ such that A is projective as a right $B$-module. Let $M$ be an $A$-module and $N$ be a $B$-module. Then $$\emph{Ext}_B^n(N,M)\cong \emph{Ext}_A^n(A\ot_BN,M). $$ 
 \end{lemma}
 
 \begin{proof}
 Let $P_{\bullet}\to N$ be a $B$ projective resolution of $N$. Then $A\ot_BP_n$ is projective as A-module so that $A\ot_BP_{\bullet}\to A\ot_BN$ is a projective resolution of $A\ot_BN$ as an $A$-module. Let $$\sigma: \text{Hom}_B(P_n,M)\to \text{Hom}_A(A\ot_BP_n,M) \text{ defined by } \sigma(f)(a\ot_B p)=af(p),$$
$$\tau: \text{Hom}_A(A\ot_BP_n,M)\to \text{Hom}_B(P_n,M) \text{ defined by } \tau(g)(p)=g(1\ot_B p)$$ 
where $a\in A,p\in P_n,f\in \text{Hom}_B(P_n,M), g\in \text{Hom}_A(A\ot_BP_n,M).$ Since $\sigma$ and $\tau$ are inverse of each other and they are homomorphisms, $\text{Hom}_A(A\ot_BP_n,M)\cong \text{Hom}_B(P_n,M)$.
\end{proof}
If we replace $A$ with $A^e$, $B$ with $A$ and take $M=A,N=k$ in the Eckmann-Shapiro lemma, we have the isomorphism Ext$_{A^e}^n(A^e\ot_Ak,A)\cong$Ext$_A^n(k,A^{ad})$. We also know that $A\cong A^e\ot_A k$ \cite[Lemma 9.4.2]{HH} and the isomorphism is given by $\rho(a)=a\ot 1\ot 1$ for all $a\in A$. Therefore Ext$_{A^e}^n(A,A)\cong$Ext$_{A^e}^n(A^e\ot_Ak,A)\cong$Ext$_A^n(k,A^{ad})$.
 
 We already have the Gerstenhaber bracket $[,]_X$ on Ext$_{A^e}^n(A^e\ot_Ak,A)$. Hence we can use the isomorphisms $\sigma$ and $\tau$ in Eckmann-Shapiro Lemma and find the bracket expression on H$^*(A,A^{ad})$. Now let $\tilde{f}\in$Hom$_{A}(P_m,A^{ad})$ and $\tilde{g}\in$Hom$_{A}(P_n,A^{ad})$. Then $[\tilde{f},\tilde{g}]_{P}\in$Hom$_{A}(P_{m+n-1},A^{ad})$ and we have 
 \begin{align*}
 [\tilde{f},\tilde{g}]_P&=\tau[\sigma(\tilde{f}),\sigma(\tilde{g})]_X\\
 &=\tau((\psi^*(\sigma(\tilde{f}))\circ\psi^*(\sigma(\tilde{g})))\theta)-(-1)^{(m-1)(n-1)}\tau((\psi^*(\sigma(\tilde{g}))\circ \psi^*(\sigma(\tilde{f})))\theta).
 \end{align*}
 For simplification we define $$\tilde{f}\circ_{P}\tilde{g}:=\tau((\psi^*(\sigma(\tilde{f}))\circ\psi^*(\sigma(\tilde{g})))\theta).$$
Then by using previous circle product formula we obtain:
 \begin{align*}
 \tilde{f}&\circ_{P}\tilde{g}(1\ot c^1\ot c^2\ot\cdots\ot c^{m+n-1})\\
 =&\tau((\psi^*(\sigma(\tilde{f}))\circ\psi^*(\sigma(\tilde{g})))\theta)(1\ot c^1\ot c^2\ot\cdots\ot c^{m+n-1})\\
 =&(\psi^*(\sigma(\tilde{f}))\circ\psi^*(\sigma(\tilde{g})))\theta((1\ot 1) \ot_A 1\ot c^1\ot c^2\ot\cdots\ot c^{m+n-1})\\
 =&\sum \sum_{i=1}^{m}(-1)^{(n-1)(i-1)}\sigma(\tilde{f})(1\otimes  c_{2}^{1}c_{2}^{2}\cdots c_{2}^{i-1}c_{2}^{*} c_{2}^{i+n}\cdots c_{2}^{m+n-1}S(c_3^1c_3^2\cdots c_3^{m+n-1})\\
 &\ot_A 1\ot c_{1}^{1}\ot c_{1}^{2}\ot \cdots \ot c_{1}^{i-1}\ot c_{1}^{*}\ot  c_{1}^{i+n}\ot \cdots \ot c_{1}^{m+n-1})\\
 =&\sum \sum_{i=1}^{m}(-1)^{(n-1)(i-1)}\tilde{f}(1\ot c_{1}^{1}\ot c_{1}^{2}\ot \cdots \ot c_{1}^{i-1}\ot c_{1}^{*}\ot  c_{1}^{i+n}\ot \cdots \ot c_{1}^{m+n-1})\\
 &c_{2}^{1}c_{2}^{2}\cdots c_{2}^{i-1}c_{2}^{*} c_{2}^{i+n}\cdots c_{2}^{m+n-1}S(c_3^1c_3^2\cdots c_3^{m+n-1}))\\
 \end{align*}
 with $\Delta(c^*)=\sum c_{1}^*\ot c_{2}^*$ and 
 \begin{align*}
c^*&=\sum \sigma(\tilde{g})(1\ot c_{2}^{i}c_2^{i+1}\cdots c_{2}^{i+n-1}\ot_A 1\ot c_{1}^{i}\ot c_{1}^{i+1}\ot \cdots \ot c_{1}^{i+n-1} )\\
&=\sum (1\ot c_{2}^{i}c_2^{i+1}\cdots c_{2}^{i+n-1})\tilde{g}(1\ot c_{1}^{i}\ot c_{1}^{i+1}\ot \cdots \ot c_{1}^{i+n-1})\\
&=\sum\tilde{g}(1\ot c_{1}^i\ot c_{1}^{i+1}\ot \cdots \ot c_{1}^{i+n-1})c_{2}^{i}c_2^{i+1}\cdots c_{2}^{i+n-1}.
 \end{align*} 
 
 We now have the Lie bracket $[,]_P$ on H$^*(A,A^{ad})$. Next, we embed H$^*(A,k)$ into H$^*(A,A^{ad})$ \cite[Corollary 9.4.7]{HH} via the unit map $$\eta_*:\text{Hom}_A(P_{\bullet},k)\to \text{Hom}_A(P_{\bullet},A^{ad}).$$ Let $f\in \text{Hom}_A(P_m,k)$ and $g\in \text{Hom}_A(P_n,k)$. Then by using counit map $$\varepsilon_*:\text{Hom}_A(P_{\bullet},A) \to \text{Hom}_A(P_{\bullet},k),$$ $\eta_*$ and bracket on H$^*(A,A^{ad})$, we derive the formula for $[f,g]\in \text{Hom}_A(P_{m+n-1},k)$: 
 $$[f,g]=\varepsilon_*[\eta_*(f),\eta_*(g)]_P=\varepsilon_*(\eta_*(f)\circ_P\eta_*(g))-(-1)^{(m-1)(n-1)}\varepsilon_*(\eta_*(g)\circ_P\eta_*(f))$$
 where
 \begin{align*}
 \varepsilon_*&((\eta_*(f)\circ_P\eta_*(g))(1\ot c^1\ot c^2\ot\cdots\ot c^{m+n-1}))\\
 =&\varepsilon(\sum \sum_{i=1}^{m}(-1)^{(n-1)(i-1)}\eta(f(1\ot c_{1}^{1}\ot c_{1}^{2}\ot \cdots \ot c_{1}^{i-1}\ot c_{1}^{*}\ot  c_{1}^{i+n}\ot \cdots \ot c_{1}^{m+n-1}))\\
 &c_{2}^{1}c_{2}^{2}\cdots c_{2}^{i-1}c_{2}^{*} c_{2}^{i+n}\cdots c_{2}^{m+n-1}S(c_3^1c_3^2\cdots c_3^{m+n-1}))
 \end{align*}
 with
 \begin{align*}
 \Delta(c^*)&=\sum c_{1}^*\ot c_{2}^* \text{ and }\\
 c^*&=\sum\eta(g(1\ot c_{1}^i\ot c_{1}^{i+1}\ot \cdots \ot c_{1}^{i+n-1}))c_{2}^{i}c_2^{i+1}\cdots c_{2}^{i+n-1}.
 \end{align*} 
 
 Therefore, the last formula is a general expression of the Gerstenhaber bracket on a Hopf algebra cohomology which is indeed inherited from the formula of the bracket on Hochschild cohomology.
 
\section*{Acknowledgement}
The author would like to thank S. Witherspoon for her precious time, suggestions and support.

\Addresses
\end{document}